\documentclass[11pt]{amsproc}

\usepackage{amsmath}
\usepackage{amsthm}
\usepackage{amssymb}
\usepackage{mathtools}
\usepackage{xcolor}
\usepackage{fullpage}
\usepackage[hypertexnames=false]{hyperref}
\usepackage{enumitem}
\usepackage{aliascnt}

\usepackage[capitalize,nameinlink,noabbrev,nosort]{cleveref}
\hypersetup{
	colorlinks=true,       
	linkcolor=orange,         
	citecolor=orange,        
	filecolor=orange,      
	urlcolor=orange,           
}

\theoremstyle{plain}
\newtheorem{theorem}{Theorem}[section]

\newaliascnt{lemma}{theorem}
\newtheorem{lemma}[lemma]{Lemma}
\aliascntresetthe{lemma}

\newaliascnt{corollary}{theorem}
\newtheorem{corollary}[corollary]{Corollary}
\aliascntresetthe{corollary}

\theoremstyle{remark}
\newaliascnt{remark}{theorem}
\newtheorem{remark}[remark]{Remark}
\aliascntresetthe{remark}

\numberwithin{equation}{section}

\newcommand{\bbF}{\mathbb{F}}
\newcommand{\bbZ}{\mathbb{Z}}
\newcommand{\bbQ}{\mathbb{Q}}

\newcommand{\calA}{\mathcal{A}}
\newcommand{\calC}{\mathcal{C}}
\newcommand{\calP}{\mathcal{P}}

\newcommand{\bsp}{\boldsymbol{p}}

\begin{document}

\title{On the $\mathcal{A}$-transcendence of a Champernowne-type constant}

\author{Shin-ichiro Seki}
\address{Nagahama Institute of Bio-Science and Technology, 1266, Tamura, Nagahama, Shiga, 526-0829, Japan}
\email{s\_seki@nagahama-i-bio.ac.jp}

\begin{abstract}
Let $p_n$ denote the $n$th prime.
We prove that for every nonzero polynomial $f(x)\in\bbZ[x]$, there exist infinitely many positive integers $n$ such that $p_n\nmid f(n)$.
\end{abstract}
\maketitle
\section{Introduction}
Let $p_n$ denote the $n$th prime.
In this note, we prove the following.
\begin{theorem}\label{thm:main}
For every nonzero polynomial $f(x)\in\bbZ[x]$, we have
\[
\#\{n\in\bbZ_{>0} : p_n\nmid f(n)\}=\infty.
\]
\end{theorem}
This is a new divisibility property of polynomial values.
One motivation for \cref{thm:main} comes from the smoothness of polynomial values.
Recall that a nonzero integer is called \emph{$y$-smooth} if each of its prime factors is at most $y$.
The following theorem gives a remarkable property of values of quadratic polynomials.
\begin{theorem}[{Bober--Fretwell--Martin--Wooley \cite[Corollary~1.2]{BoberFretwellMartinWooley2020}}]
Let $f(x)\in\mathbb{Z}[x]$ be quadratic.
Then, for every $\varepsilon>0$, there exist infinitely many positive integers $n$ such that $f(n)$ is $n^{\varepsilon}$-smooth.
\end{theorem}
If this theorem could be extended to the case where $\deg f\geq 3$, then \cref{thm:main} would follow as a special case, since the prime number theorem gives $p_n\sim n\log n$.
However, such an extension appears difficult to establish at present.

Besides its connection with the smoothness of polynomial values, another motivation for considering \cref{thm:main} comes from arithmetic in the ring
\[
\calA\coloneqq\Biggl(\prod_{p:\text{ prime}}\bbZ/p\bbZ\Biggr) \Bigg/ \Biggl(\bigoplus_{p:\text{ prime}}\bbZ/p\bbZ\Biggr),
\]
which has attracted increasing interest in recent years.
Indeed, \cref{thm:main} can be reformulated as a transcendence statement concerning a certain element of $\calA$.

Let $\calP_{\calA}^0$ denote the set of all finite algebraic numbers introduced by Rosen in \cite{Rosen}.
Let $\calC_{\calA}$ be the integral closure of $\bbQ$ in $\calA$.
Then the strict inclusions $\bbQ\subsetneq\calP_{\calA}^0\subsetneq\calC_{\calA}\subsetneq\calA$ hold.
In the ring $\mathcal A$, two notions of transcendence have been studied: transcendence in the sense of not belonging to $\mathcal P_{\mathcal A}^0$, and transcendence in the sense of not belonging to $\mathcal C_{\mathcal A}$.
Let $\alpha=(a_p\bmod p)_p\in\mathcal A$.
It is easy to see that if $a_p\to\infty$ as $p\to\infty$ and $a_p$ grows more slowly than every positive power of $p$, then $\alpha$ is transcendental in the stronger sense (\cite[Lemma~2.3]{MatsusakaSeki}); that is, $\alpha\notin\mathcal C_{\mathcal A}$.
For example, $(\lfloor\log p\rfloor\bmod{p})_p\notin\calC_{\calA}$ (\cite[Example~2.4]{MatsusakaSeki}).
Here, $\lfloor X\rfloor$ denotes the greatest integer less than or equal to $X$.
In contrast, the following criterion, established by Luca and Zudilin using a prime-counting argument, makes it possible to prove transcendence even when $a_p$ grows substantially faster.
\begin{lemma}[{Luca--Zudilin's second criterion \cite{LucaZudilin}, \cite[Lemma~2.5]{MatsusakaSeki}}]\label{lem:LucaZudilin}
Let $\alpha = (a_p\bmod{p})_p \in \mathcal{A}$.
Let $g(X)$ and $h(X)$ be positive-valued, monotonically increasing functions satisfying $g(X)\log X=o(h(X))$ as $X\to\infty$.
Let $S$ be an infinite set of primes, and let $(b_p)_{p \in S}$ be a sequence of integers satisfying $b_p\to\infty$ and $b_p = O(g(p))$ as $p \to \infty$.
Suppose that $a_p\equiv b_p \pmod{p}$ for all $p \in S$. 
Assume further that $\#\{ p \le X : p \in S \} \gg h(X)$ for all sufficiently large $X$. 
Then $\alpha\notin\calC_{\calA}$.
\end{lemma}
Although this formulation allows more general growth functions than
\cite[Lemma~2.5]{MatsusakaSeki}, the same counting argument applies.
This criterion shows that an element is transcendental if, on a sufficiently dense set of primes, it admits integer representatives $b_p$ tending to infinity and satisfying $b_p=O(p^\epsilon)$ for some $0<\epsilon<1$.
For example, one obtains
\[
(\lfloor\sqrt{p}\rfloor\bmod{p})_p\in\calA\setminus\calC_{\calA}
\]
(\cite[Example~2.6]{MatsusakaSeki}).
The strength of this criterion lies in the fact that, beyond such artificial examples, it can also establish the transcendence of more natural elements, such as
\[
\bigl((p+1-\#E(\bbF_p))\bmod{p}\bigr)_p\in\calA\setminus\calC_{\calA},
\]
associated with an elliptic curve $E$ defined over $\bbQ$, and
\[
(B_{\frac{p+1}{2}}\bmod{p})_p\in\calA\setminus\calC_{\calA},
\]
defined in terms of the Bernoulli numbers $(B_n)_n$.
See \cite[Theorem~4.1 and 5.1]{MatsusakaSeki}.

To examine how far the growth of $g(X)$ can be extended beyond $X^{\epsilon}$, note that, by the prime number theorem, $h(X)$ can be at most of order $X/\log X$.
Thus, the criterion allows the possibility of taking $g(X)=X/(\log X)^3$, but it cannot be applied with $g(X)=X/\log X$.
Therefore, writing $\pi(X)$ for the prime-counting function, the criterion shows that
\[
\pi^3(\bsp)\coloneqq(\pi(\pi(\pi(p)))\bmod{p})_p\in\mathcal{A}\setminus\calC_{\calA}.
\]
However, it does not apply to
\[
\pi(\bsp)\coloneqq(\pi(p)\bmod{p})_p=(n\bmod{p_n})_{p_n}\in\mathcal{A}.
\]

The proof of \cite[Lemma~2.5]{MatsusakaSeki} uses the crude estimate $\omega(n)=O(\log n)$, where $\omega(n)$ denotes the number of distinct prime divisors of $n$.
Using instead the standard sharper estimate $\omega(n)=O\left(\frac{\log n}{\log\log n}\right)$, the same counting argument shows that the hypothesis $g(X)\log X=o(h(X))$ may be weakened to $g(X)\frac{\log X}{\log\log X}=o(h(X))$.
This makes the choice $g(X)=X/(\log X)^2$ admissible and hence would imply that
\[
\pi^2(\bsp)\coloneqq(\pi(\pi(p))\bmod{p})_p\in\mathcal{A}\setminus\calC_{\calA}.
\]
However, it would still not allow one to take $g(X)=X/\log X$.
If, in each component of $\pi(\bsp)$, $\pi^2(\bsp)$, and $\pi^3(\bsp)$, we omit ``$\mathrm{mod} \ p$'' and write only the chosen integer representatives, then
\begin{align*}
\pi(\bsp)&=(1,2,3,4,5,6,7,8,9,10,11,12,13,14,15,16,17,18,19,20,\dots),\\
\pi^2(\bsp)&=(0,1,2,2,3,3,4,4,4,4,5,5,6,6,6,6,7,7,8,8,\dots),\\
\pi^3(\bsp)&=(0,0,1,1,2,2,2,2,2,2,3,3,3,3,3,3,4,4,4,4,\dots).
\end{align*}
As noted in \cite{MatsusakaSeki}, the element $\pi(\bsp)$ is reminiscent of the Champernowne constant.
The classical Champernowne constant was proved transcendental by Mahler \cite{Mahler} using Schneider's theorem.

The Champernowne-type constant $\pi(\bsp)$ was originally conceived by the author several years ago as an artificial example whose irrationality in $\calA$ follows easily from the prime number theorem.
Although \cref{lem:LucaZudilin} subsequently provided a new method for proving the transcendence of several natural elements of $\calA$, $\pi(\bsp)$ remained an example for which transcendence was difficult to establish.

The transcendence of $\pi(\bsp)$ in the strong sense, namely, $\pi(\bsp)\notin\calC_{\calA}$, was conjectured in \cite[Conjecture~6.2]{MatsusakaSeki}, and the following partial results were obtained:
\begin{itemize}[leftmargin=2.5em]
\item The transcendence of $\pi(\bsp)$ in the weak sense holds; that is, $\pi(\bsp)\notin\calP_{\calA}^0$ (\cite[Proposition~6.1]{MatsusakaSeki}).
\item For every nonzero polynomial $f(x)\in\bbZ[x]$ such that every irreducible factor of $f(x)$ has degree at most two, one has $f(\pi(\bsp))\neq0$ (\cite[Theorem~6.7]{MatsusakaSeki}).
\item Assuming either the conjecture on the smoothness of polynomial values or the conjecture on the equidistribution of the roots of every irreducible polynomial, one has $\pi(\bsp)\notin\calC_{\calA}$ (\cite[Proposition~6.5 and the paragraph immediately preceding Section~7]{MatsusakaSeki}).
\end{itemize}
Since it is easy to verify that $\pi(\bsp)\notin\calC_{\calA}$ is equivalent to \cref{thm:main}, we obtain the following.
\begin{corollary}
The Champernowne-type constant $\pi(\bsp)$ is transcendental in the strong sense; that is, $\pi(\bsp)\notin\calC_{\calA}$.
\end{corollary}
In \cite{MatsusakaSeki}, the authors ask whether Conjecture~6.2 can be proved using more specific properties of $p_n$, rather than relying only on its asymptotic behavior $p_n\sim n\log n$.
The proof of \cref{thm:main} follows precisely this direction: besides the prime number theorem, it uses the Maynard--Tao theorem on small gaps between primes.
\subsection*{Use of AI}
The author used ChatGPT, powered by GPT-5.6 Sol Pro, during the
development of this work.
The key idea and an initial version of the proof of \cref{thm:main} were suggested by the model.
The author also used ChatGPT to improve the English prose of the manuscript. The author independently reconstructed and verified the complete argument and rewrote the proof presented in this note.
The author takes full responsibility for all claims, proofs, and text in the note.
\subsection*{Acknowledgements}
The author thanks Prof.~Toshiki Matsusaka for reading the manuscript and providing helpful comments.
The author also thanks Prof.~Yuta Suzuki for discussions on Luca--Zudilin's second criterion.
This research was supported by JSPS KAKENHI Grant Number JP26K06734.
\section{Proof of the main theorem}
\begin{proof}
Let $f(x)\in\bbZ[x]$ be an arbitrary nonzero polynomial, and let $d$ denote its degree.
Assume that, for every sufficiently large integer $n$, one has $p_n\mid f(n)$.
Then we have $d\geq 1$.
Choose an integer $N$ such that $p_n\mid f(n)$ for every integer $n\geq N$.
Set $k\coloneqq d(d+1)$.
By the Maynard--Tao theorem \cite[Theorem~1.1]{Maynard},
\[
\liminf_{n\to\infty}(p_{n+k}-p_n)<\infty.
\]
It therefore follows from the pigeonhole principle that there exist integers $0=h_0<h_1<\cdots<h_k$ such that
\begin{equation}\label{eq:MaynardTao}
p_{n+j}=p_n+h_j,\qquad 0\leq j\leq k
\end{equation}
holds for infinitely many integers $n$.
Fix one such tuple $(h_j)_j$, and let $M$ denote the infinite set of all integers $n\geq N$ for which \eqref{eq:MaynardTao} holds.

The homogeneous system of $k$ linear equations in the $k+1$ variables $x_0,\dots,x_k$ given by
\begin{equation}\label{eq:Artin}
\sum_{j=0}^k j^u h_j^v x_j=0,\qquad 0\leq u\leq d,\quad 0\leq v\leq d-1
\end{equation}
has a nontrivial integer solution
\[
(x_0,\dots,x_k)=(c_0,\dots,c_k)\in\bbZ^{k+1}\setminus\{(0,\dots,0)\}.
\]
Fix one such solution.
Here, we adopt the convention that $0^0=1$.

Using $f$ and the chosen tuples $(h_j)_j$ and $(c_j)_j$, define the polynomial $\Phi(x,y)$ in two variables with integer coefficients by
\[
\rho(y)\coloneqq\prod_{i=0}^k(y+h_i),
\]
and
\[
\Phi(x,y)\coloneqq\sum_{j=0}^kc_jf(x+j)\frac{\rho(y)}{y+h_j}\in\bbZ[x,y].
\]
Suppose, for contradiction, that $\Phi(x,y)$ is the zero polynomial.
Then, for every $0\leq i\leq k$, we have
\[
0=\Phi(x,-h_i)=c_if(x+i)\prod_{j\neq i}(h_j-h_i).
\]
Since the $h_j$ are pairwise distinct and $f$ is nonzero, it follows that $c_i=0$.
Thus, $c_i=0$ for every $0\leq i\leq k$, contradicting the fact that $(c_j)_j$ is a nontrivial solution.
Therefore, $\Phi(x,y)$ is a nonzero polynomial.
Write
\[
\Phi(x,y)=\sum_{i,j}a_{i,j}x^iy^j,
\]
and let $(I,J)$ be a pair such that $a_{I,J}\neq0$, $I+J$ is maximal among all pairs $(i,j)$ with $a_{i,j}\neq0$, and, subject to this condition, $J$ is maximal.
Then, by the prime number theorem $p_n\sim n\log n$, we have
\begin{equation}\label{eq:asymp_Phi}
\Phi(n,p_n)\sim a_{I,J}n^{I+J}(\log n)^J
\qquad\text{as }n\to\infty.
\end{equation}

In what follows, $n$ is always assumed to belong to $M$.
Then, by the definition of $M$, we have
\begin{equation}\label{eq:integer}
\frac{\Phi(n,p_n)}{\rho(p_n)}=\sum_{j=0}^kc_j\frac{f(n+j)}{p_n+h_j}=\sum_{j=0}^kc_j\frac{f(n+j)}{p_{n+j}}\in\bbZ.
\end{equation}
We now apply the finite geometric-series identity
\[
\frac{1}{p_n+h_j}=\sum_{v=0}^{d-1}\frac{(-h_j)^v}{p_n^{v+1}}+\frac{(-h_j)^d}{p_n^d(p_n+h_j)}.
\]
For each $0\leq j\leq k$, write
\[
f(n+j)=\sum_{u=0}^d b_u(n)j^u.
\]
Then, using \eqref{eq:Artin}, we compute
\[
\sum_{j=0}^kc_jf(n+j)\sum_{v=0}^{d-1}\frac{(-h_j)^v}{p_n^{v+1}}=\sum_{u=0}^d\sum_{v=0}^{d-1}\frac{(-1)^vb_u(n)}{p_n^{v+1}}\sum_{j=0}^kc_jj^uh_j^v=0.
\]
Therefore, we obtain
\[
\frac{\Phi(n,p_n)}{\rho(p_n)}=\sum_{j=0}^kc_jf(n+j)\frac{(-h_j)^d}{p_n^d(p_n+h_j)}=O(n^d/p_n^{d+1})=o(1)\qquad\text{as }n\to\infty.
\]
Here, the implied constant depends only on $f$, $(h_j)_j$, and $(c_j)_j$, and is independent of $n$.
Combining this with \eqref{eq:integer}, we conclude that
\[
\Phi(n,p_n)=0
\]
for all sufficiently large $n\in M$.
This contradicts \eqref{eq:asymp_Phi}.
\end{proof}
\begin{remark}
It is noteworthy that the proof has a structure reminiscent of classical arguments in transcendental number theory: integrality is combined with an asymptotic estimate to derive a contradiction.
In the proof above, the Maynard--Tao theorem is used together with a nontrivial solution of a homogeneous system of linear equations to obtain the estimate $O(n^d/p_n^{d+1})$.
(In fact, the weaker estimate $O(n^d/p_n^d)=o(1)$ would already suffice.)
For this argument, Zhang's earlier result $\displaystyle\liminf_{n\to\infty}(p_{n+1}-p_n)<\infty$ is not sufficient.
What is needed is the stronger Maynard--Tao result that, for every fixed $k$, one has $\displaystyle\liminf_{n\to\infty}(p_{n+k}-p_n)<\infty$.
\end{remark}
The proof presented here suggests that fine local information on the distribution of primes, going beyond the asymptotic information provided by the prime number theorem, can be effectively combined with a linear-algebraic annihilation argument to address arithmetic questions concerning polynomial values.
We hope that this perspective will lead to further progress both in the study of polynomial values and in transcendence problems in $\calA$.

\end{document}